\documentclass[reqno,11pt]{article}
\usepackage{a4wide}
\usepackage{color,eucal,enumerate,mathrsfs}
\usepackage[normalem]{ulem}
\usepackage{mdwlist}
\usepackage{amsmath}
\usepackage{amssymb,epsfig,bbm}
\numberwithin{equation}{section}
\usepackage{hyperref}
\usepackage{hyphenat}
\usepackage{cite}
\usepackage[latin1]{inputenc}

\newcommand{\N}{\mathbb{N}}

\newcommand{\R}{\mathbb{R}}
\newcommand{\Z}{\mathbb{Z}}

\newcommand{\mm}{{\mbox{\boldmath$m$}}}

\newcommand{\sfd}{{\sf d}}

\newcommand{\Kliminf}{K\kern-3pt-\kern-2pt\mathop{\rm lim\,inf}\limits}  
\renewcommand{\d}{{\mathrm d}}

\newcommand{\restr}[1]{\lower3pt\hbox{$|_{#1}$}}

\newcommand{\eps}{\varepsilon}  
\newcommand{\nchi}{{\raise.3ex\hbox{$\chi$}}}
\newcommand{\weakto}{\rightharpoonup}

\newcommand{\lims}{\varlimsup}

\newcommand{\fr}{\penalty-20\null\hfill$\blacksquare$}                      

\newcommand{\X}{{\rm X}}

\renewcommand{\mm}{\mathfrak m}                                

\newenvironment{proof}{\removelastskip\par\medskip   
\noindent{\textit{Proof.}}\rm}{\penalty-20\null\hfill$\square$\par\medbreak}

\newtheorem{theorem}{Theorem}[section]

\newtheorem{lemma}[theorem]{Lemma}
\newtheorem{proposition}[theorem]{Proposition}

\newtheorem{definition}[theorem]{Definition}

\newtheorem{remark}[theorem]{Remark}

\newcommand{\beq}{\begin{equation}}
\newcommand{\eeq}{\end{equation}}

\title{Differential structure associated to axiomatic Sobolev spaces}
\author{Nicola Gigli\thanks{\textsc{SISSA, Via Bonomea 265, 34136 Trieste.}
\textit{E-mail address:} \textsf{ngigli@sissa.it}}
\and Enrico Pasqualetto\thanks{\textsc{SISSA, Via Bonomea 265, 34136 Trieste.}
\textit{E-mail address:} \textsf{epasqual@sissa.it}}}

\begin{document}
\maketitle
\begin{abstract} 
The aim of this note is to explain in which sense an axiomatic Sobolev space over a
general metric measure space (\`{a} la Gol'dshtein-Troyanov) induces -- under
suitable locality assumptions -- a first-order differential structure.
\end{abstract}
\bigskip

\medskip

\tableofcontents
\section*{Introduction}
\addcontentsline{toc}{section}{Introduction}
An axiomatic approach to the theory of Sobolev spaces over abstract metric
measure spaces has been proposed by V.\ Gol'dshtein and M.\ Troyanov in
\cite{GoldTroy01}. Their construction covers many important notions:
the weighted Sobolev space on a Riemannian manifold,
the Haj{\l}asz Sobolev space \cite{Hajlasz1996} and the
Sobolev space based on the concept of upper gradient
\cite{heinonen1998,Cheeger00,Shanmugalingam00,AmbrosioGigliSavare11}.

A key concept in \cite{GoldTroy01} is the so-called \emph{$D$-structure}:
given a metric measure space $(\X,\sfd,\mm)$ and an exponent $p\in(1,\infty)$,
we associate to any function $u\in L^p_{loc}(\X)$ a family $D[u]$ of non-negative Borel
functions called \emph{pseudo-gradients}, which exert some control from above
on the variation of $u$. The pseudo-gradients are not explicitly specified,
but they are rather supposed to fulfil a list of axioms.
Then the space $W^{1,p}(\X,\sfd,\mm,D)$ is defined as the set of all
functions in $L^p(\mm)$ admitting a pseudo-gradient in $L^p(\mm)$.
By means of standard functional analytic techniques, it is possible
to associate to any Sobolev function $u\in W^{1,p}(\X,\sfd,\mm,D)$ a
uniquely determined minimal object $\underline D u\in D[u]\cap L^p(\mm)$,
called \emph{minimal pseudo-gradient} of $u$.

More recently, the first author of the present paper introduced a differential
structure on general metric measure spaces (cf.\ \cite{Gigli14,Gigli17}).
The purpose was to develop a second-order differential calculus on spaces
satisfying lower Ricci curvature bounds (or briefly, $\sf RCD$ spaces;
we refer to \cite{Ambrosio18,Villani2016,Villani2017} for a presentation
of this class of spaces).
The fundamental tools for this theory are the $L^p$-normed $L^\infty$-modules,
among which a special role is played by the \emph{cotangent module}, denoted
by $L^2(T^*\X)$. Its elements can be thought of as `measurable $1$-forms on $\X$'.

The main result of this paper -- namely Theorem \ref{thm:cot_mod} -- says that
any $D$-structure (satisfying suitable locality properties) gives rise to a natural
notion of cotangent module $L^p(T^*\X;D)$, whose properties are analogous to the ones
of the cotangent module $L^2(T^*\X)$ described in \cite{Gigli14}.
Roughly speaking, the cotangent module allows us to represent minimal
pseudo-gradients as pointwise norms of suitable linear objects.
More precisely, this theory provides the existence of an abstract differential
$\d:\,W^{1,p}(\X,\sfd,\mm,D)\to L^p(T^*\X;D)$, which is a linear operator
such that the pointwise norm $|\d u|\in L^p(\mm)$ of $\d u$ coincides with
$\underline D u$ in the $\mm$-a.e.\ sense for any
function $u\in W^{1,p}(\X,\sfd,\mm,D)$.
\section{General notation}
For the purpose of the present paper, a \emph{metric measure space}
is a triple $(\X,\sfd,\mm)$, where
\begin{equation}\begin{split}
(\X,\sfd)&\qquad\text{is a complete and separable metric space,}\\
\mm\neq 0&\qquad\text{is a non-negative Borel measure on }\X\text{, finite on balls.}
\end{split}\end{equation}
Fix $p\in[1,\infty)$. Several functional spaces over $\X$
will be used in the forthcoming discussion:
\begin{align*}
L^0(\mm):&\quad\text{ the Borel functions }u:\,\X\to\R
\text{, considered up to }\mm\text{-a.e.\ equality.}\\
L^p(\mm):&\quad\text{ the functions }u\in L^0(\mm)\text{ for which }
|u|^p\text{ is integrable.}\\
L^p_{loc}(\mm):&\quad\text{ the functions }u\in L^0(\mm)
\text{ with }u\restr B\in L^p\big(\mm\restr B\big)\text{ for any }
B\subseteq\X\text{ bounded Borel.}\\
L^\infty(\mm):&\quad\text{ the functions }u\in L^0(\mm)\text{ that are
essentially bounded.}\\
L^0(\mm)^+:&\quad\text{ the Borel functions }u:\,\X\to[0,+\infty]
\text{, considered up to }\mm\text{-a.e.\ equality.}\\
L^p(\mm)^+:&\quad\text{ the functions }u\in L^0(\mm)^+\text{ for which }
|u|^p\text{ is integrable.}\\
L^p_{loc}(\mm)^+:&\quad\text{ the functions }u\in L^0(\mm)^+
\text{ with }u\restr B\in L^p\big(\mm\restr B\big)^+\text{ for any }
B\subseteq\X\text{ bounded Borel.}\\
{\rm LIP}(\X):&\quad\text{ the Lipschitz functions }u:\X\to\R,
\text{ with Lipschitz constant denoted by }{\rm Lip}(u).\\
{\sf Sf}(\X):&\quad\text{ the functions }u\in L^0(\mm)\text{ that are simple,
i.e.\ with a finite essential image.}
\end{align*}
Observe that for any $u\in L^p_{loc}(\mm)^+$ it holds that $u(x)<+\infty$
for $\mm$-a.e.\ $x\in\X$. 
We also recall that the space ${\sf Sf}(\X)$ is strongly dense in $L^p(\mm)$
for every $p\in[1,\infty]$.
\begin{remark}{\rm
In \cite[Section 1.1]{GoldTroy01} a more general notion of $L^p_{loc}(\mm)$ is considered,
based upon the concept of \emph{$\mathcal K$-set}.
We chose the present approach for simplicity, but the following discussion would remain
unaltered if we replaced our definition of $L^p_{loc}(\mm)$ with
the one of \cite{GoldTroy01}.
\fr}\end{remark}
\section{Axiomatic theory of Sobolev spaces}
We begin by briefly recalling the axiomatic notion of Sobolev space that has been
introduced by V.\ Gol'dshtein and M.\ Troyanov in \cite[Section 1.2]{GoldTroy01}:
\begin{definition}[$D$-structure]
Let $(\X,\sfd,\mm)$ be a metric measure space. Let $p\in[1,\infty)$ be fixed.
Then a \emph{$D$-structure} on $(\X,\sfd,\mm)$ is any map $D$ associating to each function $u\in L^p_{loc}(\mm)$
a family $D[u]\subseteq L^0(\mm)^+$ of \emph{pseudo-gradients} of $u$,
which satisfies the following axioms:
\begin{itemize}
\item[\textbf{\emph{A1}}]\textbf{\emph{(Non triviality)}} It holds that
${\rm Lip}(u)\,\nchi_{\{u>0\}}\in D[u]$ for every $u\in L^p_{loc}(\mm)^+\cap{\rm LIP}(\X)$.
\item[\textbf{\emph{A2}}]\textbf{\emph{(Upper linearity)}} Let $u_1,u_2\in L^p_{loc}(\mm)$ be fixed.
Consider $g_1\in D[u_1]$ and $g_2\in D[u_2]$. Suppose that the inequality
$g\geq|\alpha_1|\,g_1+|\alpha_2|\,g_2$ holds $\mm$-a.e.\ in $\X$ for some $g\in L^0(\mm)^+$
and $\alpha_1,\alpha_2\in\R$. Then $g\in D[\alpha_1\,u_1+\alpha_2\,u_2]$.
\item[\textbf{\emph{A3}}]\textbf{\emph{(Leibniz rule)}} Fix a function $u\in L^p_{loc}(\mm)$ and
a pseudo-gradient $g\in D[u]$ of $u$. Then for every $\varphi\in{\rm LIP}(\X)$
bounded it holds that $g\,\sup_\X|\varphi|+{\rm Lip}(\varphi)\,|u|\in D[\varphi\,u]$.
\item[\textbf{\emph{A4}}]\textbf{\emph{(Lattice property)}} Fix $u_1,u_2\in L^p_{loc}(\mm)$.
Given any $g_1\in D[u_1]$ and $g_2\in D[u_2]$, one has that
$\max\{g_1,g_2\}\in D\big[\max\{u_1,u_2\}\big]\cap D\big[\min\{u_1,u_2\}\big]$.
\item[\textbf{\emph{A5}}]\textbf{\emph{(Completeness)}} Consider two sequences
$(u_n)_n\subseteq L^p_{loc}(\mm)$ and $(g_n)_n\subseteq L^p(\mm)$ that
satisfy $g_n\in D[u_n]$ for every $n\in\N$. Suppose that
there exist $u\in L^p_{loc}(\mm)$ and $g\in L^p(\mm)$ such that $u_n\to u$ in $L^p_{loc}(\mm)$
and $g_n\to g$ in $L^p(\mm)$. Then $g\in D[u]$.
\end{itemize}
\end{definition}
\begin{remark}\label{rmk:conseq_D-structure}{\rm
It follows from axioms \textbf{A1} and \textbf{A2} that $0\in D[c]$ for every constant map $c\in\R$.
Moreover, axiom \textbf{A2} grants
that the set $D[u]\cap L^p(\mm)$ is convex and that $D[\alpha\,u]=|\alpha|\,D[u]$
for every $u\in L^p_{loc}(\mm)$ and $\alpha\in\R\setminus\{0\}$,
while axiom \textbf{A5} implies that each set $D[u]\cap L^p(\mm)$ is closed
in the space $L^p(\mm)$.
\fr}\end{remark}

Given any Borel set $B\subseteq\X$, we define the \emph{$p$-Dirichlet energy} of a map
$u\in L^p(\mm)$ on $B$ as
\begin{equation}\label{eq:p-Dirichlet_energy}
\mathcal{E}_p(u|B):=\inf\left\{\int_B g^p\,\d\mm\;\bigg|\;g\in D[u]\right\}\in[0,+\infty].
\end{equation}
For the sake of brevity, we shall use the notation $\mathcal{E}_p(u)$ to indicate $\mathcal{E}_p(u|\X)$.
\begin{definition}[Sobolev space]
Let $(\X,\sfd,\mm)$ be a metric measure space. Let $p\in[1,\infty)$ be fixed. Given a $D$-structure
on $(\X,\sfd,\mm)$, we define the \emph{Sobolev class} associated to $D$ as
\begin{equation}\label{eq:Sobolev_class_associated_to_D-structure}
{\rm S}^p(\X)={\rm S}^p(\X,\sfd,\mm,D):=
\big\{u\in L^p_{loc}(\mm)\;:\;\mathcal{E}_p(u)<+\infty\big\}.
\end{equation}
Moreover, the \emph{Sobolev space} associated to $D$ is defined as
\begin{equation}\label{eq:Sobolev_space_associated_to_D-structure}
W^{1,p}(\X)=W^{1,p}(\X,\sfd,\mm,D):=L^p(\mm)\cap{\rm S}^p(\X,\sfd,\mm,D).
\end{equation}
\end{definition}
\begin{theorem}
The space $W^{1,p}(\X,\sfd,\mm,D)$ is a Banach space if endowed with the norm
\begin{equation}\label{eq:Sobolev_space_associated_to_D-structure_norm}
{\|u\|}_{W^{1,p}(\X)}:=\left({\|u\|}^p_{L^p(\mm)}+\mathcal{E}_p(u)\right)^{1/p}
\quad\text{ for every }u\in W^{1,p}(\X).
\end{equation}
\end{theorem}

For a proof of the previous result, we refer to \cite[Theorem 1.5]{GoldTroy01}.
\begin{proposition}[Minimal pseudo-gradient]
Let $(\X,\sfd,\mm)$ be a metric measure space and let $p\in(1,\infty)$.
Consider any $D$-structure on $(\X,\sfd,\mm)$. Let $u\in{\rm S}^p(\X)$ be given.
Then there exists a unique element $\underline{D}u\in D[u]$,
which is called the \emph{minimal pseudo-gradient} of $u$, such that
$\mathcal{E}_p(u)={\|\underline{D}u\|}^p_{L^p(\mm)}$.
\end{proposition}

Both existence and uniqueness of the minimal pseudo-gradient follow from the
fact that the set $D[u]\cap L^p(\mm)$ is convex and closed by Remark
\ref{rmk:conseq_D-structure} and that the space $L^p(\mm)$ is uniformly convex;
see \cite[Proposition 1.22]{GoldTroy01} for the details.

In order to associate a differential structure to an axiomatic Sobolev space,
we need to be sure that the pseudo-gradients of a function depend only on the local
behaviour of the function itself, in a suitable sense. For this reason, we
propose various notions of locality:
\begin{definition}[Locality]
Let $(\X,\sfd,\mm)$ be a metric measure space. Fix $p\in(1,\infty)$.
Then we define five notions of locality for $D$-structures on $(\X,\sfd,\mm)$:
\begin{itemize}
\item[\textbf{\emph{L1}}]
If $B\subseteq\X$ is Borel and $u\in{\rm S}^p(\X)$ is $\mm$-a.e.\ constant
in $B$, then $\mathcal{E}_p(u|B)=0$.
\item[\textbf{\emph{L2}}]
If $B\subseteq\X$ is Borel and $u\in{\rm S}^p(\X)$ is $\mm$-a.e.\ constant
in $B$, then $\underline{D}u=0$ $\mm$-a.e.\ in $B$.
\item[\textbf{\emph{L3}}]
If $u\in{\rm S}^p(\X)$ and $g\in D[u]$, then $\nchi_{\{u>0\}}\,g\in D[u^+]$.
\item[\textbf{\emph{L4}}]
If $u\in{\rm S}^p(\X)$ and $g_1,g_2\in D[u]$,
then $\min\{g_1,g_2\}\in D[u]$.
\item[\textbf{\emph{L5}}]
If $u\in{\rm S}^p(\X)$ then $\underline Du\leq g$ holds $\mm$-a.e.\ in $\X$
for every $g\in D[u]$.
\end{itemize}
\end{definition}
\begin{remark}{\rm
In the language of \cite[Definition 1.11]{GoldTroy01}, the properties \textbf{L1} and
\textbf{L3} correspond to \emph{locality} and \emph{strict locality},
respectively.
\fr}\end{remark}

We now discuss the relations among the several notions of locality:
\begin{proposition}\label{prop:implications_locality}
Let $(\X,\sfd,\mm)$ be a metric measure space. Let $p\in(1,\infty)$.
Fix a $D$-structure on $(\X,\sfd,\mm)$. Then the following
implications hold:
\begin{equation}\begin{split}
\textbf{\emph{L3}}&\quad\Longrightarrow\\
\textbf{\emph{L4}}&\quad\Longleftrightarrow\\
\textbf{\emph{L1}}+\textbf{\emph{L5}}&\quad\Longrightarrow
\end{split}\begin{split}
&\quad\textbf{\emph{L2}}\quad\Longrightarrow\quad\textbf{\emph{L1}},\\
&\quad\textbf{\emph{L5}}\\
&\quad\textbf{\emph{L2}}+\textbf{\emph{L3}}.
\end{split}\end{equation}
\end{proposition}
\begin{proof}\\
{\color{blue}\textbf{L2} $\Longrightarrow$ \textbf{L1}.}
Simply notice that $\mathcal E_p(u|B)\leq\int_B(\underline D u)^p\,\d\mm=0$.\\
{\color{blue}\textbf{L3} $\Longrightarrow$ \textbf{L2}.}
Take a constant $c\in\R$ such that the equality $u=c$ holds $\mm$-a.e.\ in $B$.
Given that $\underline{D}u\in D[u-c]\cap D[c-u]$ by axiom \textbf{A2} and Remark
 \ref{rmk:conseq_D-structure}, we deduce from \textbf{L3} that
\[\begin{split}
&\nchi_{\{u>c\}}\,\underline{D}u\in D\big[(u-c)^+\big],\\
&\nchi_{\{u<c\}}\,\underline{D}u\in D\big[(c-u)^+\big].
\end{split}\]
Given that $u-c=(u-c)^+ -(c-u)^+$, by applying again axiom \textbf{A2} we see that
$$\nchi_{\{u\neq c\}}\,\underline{D}u=
\nchi_{\{u>c\}}\,\underline{D}u+\nchi_{\{u<c\}}\,\underline{D}u
\in D[u-c]=D[u].$$
Hence the minimality of $\underline{D}u$ grants that
$$\int_\X(\underline{D}u)^p\,\d\mm\leq\int_{\{u\neq c\}}(\underline{D}u)^p\,\d\mm,$$
which implies that $\underline{D}u=0$ holds $\mm$-a.e.\ in $\{u=c\}$, thus also $\mm$-a.e.\ in $B$.
This means that the $D$-structure satisfies the property \textbf{L2}, as required.\\
{\color{blue}\textbf{L4} $\Longrightarrow$ \textbf{L5}.} We argue by contradiction:
suppose the existence of $u\in{\rm S}^p(\X)$ and $g\in D[u]$ such that
$\mm\big(\{\underline Du>g\}\big)>0$, whence $h:=\min\{\underline D u,g\}\in L^p(\mm)$ 
satisfies $\int h^p\,\d\mm<\int(\underline D u)^p\,\d\mm$. Since $h\in D[u]$ by
\textbf{L4}, we deduce that $\mathcal E_p(u)<\int(\underline D u)^p\,\d\mm$,
getting a contradiction.\\
{\color{blue}\textbf{L5} $\Longrightarrow$ \textbf{L4}.} Since $\underline D u\leq g_1$
and $\underline D u\leq g_2$ hold $\mm$-a.e., we see that $\underline D u\leq\min\{g_1,g_2\}$ holds $\mm$-a.e.\ as well. Therefore $\min\{g_1,g_2\}\in D[u]$ by \textbf{A2}.\\
{\color{blue}\textbf{L1}+\textbf{L5} $\Longrightarrow$ \textbf{L2}+\textbf{L3}.}
Property \textbf{L1} grants the existence of $(g_n)_n\subseteq D[u]$
with $\int_B(g_n)^p\,\d\mm\to 0$. Hence \textbf{L5} tells us that
$\int_B(\underline D u)^p\,\d\mm\leq\lim_n\int_B(g_n)^p\,\d\mm=0$,
which implies that $\underline D u=0$ holds $\mm$-a.e.\ in $B$, yielding \textbf{L2}.
We now prove the validity of \textbf{L3}: it holds that $D[u]\subseteq D[u^+]$,
because we know that $h=\max\{h,0\}\in D\big[\max\{u,0\}\big]=D[u^+]$ for every $h\in D[u]$
by \textbf{A4} and $0\in D[0]$, in particular $u^+\in{\rm S}^p(\X)$.
Given that $u^+=0$ $\mm$-a.e.\ in the set $\{u\leq 0\}$, one has that $\underline D u^+=0$
holds $\mm$-a.e.\ in $\{u\leq 0\}$ by \textbf{L2}. Hence for any $g\in D[u]$
we have $\underline D u^+\leq\nchi_{\{u>0\}}\,g$ by \textbf{L5}, which implies
that $\nchi_{\{u>0\}}\,g\in D[u^+]$ by \textbf{A2}. Therefore \textbf{L3} is proved.
\end{proof}
\begin{definition}[Pointwise local]
Let $(\X,\sfd,\mm)$ be a metric measure space and $p\in(1,\infty)$.
Then a $D$-structure on $(\X,\sfd,\mm)$ is said to be \emph{pointwise local}
provided it satisfies \textbf{\emph{L1}} and \textbf{\emph{L5}} (thus
also \textbf{\emph{L2}}, \textbf{\emph{L3}} and \textbf{\emph{L4}} by Proposition
\ref{prop:implications_locality}).
\end{definition}

We now recall other two notions of locality for $D$-structures that
appeared in the literature:
\begin{definition}[Strong locality]
Let $(\X,\sfd,\mm)$ be a metric measure space and $p\in(1,\infty)$.
Consider a $D$-structure on $(\X,\sfd,\mm)$. Then we give the following definitions:
\begin{itemize}
\item[$\rm i)$] We say that $D$ is \emph{strongly local in the sense
of Timoshin} provided
\begin{equation}\label{eq:loc_Timosh}
\nchi_{\{u_1<u_2\}}\,g_1+\nchi_{\{u_2<u_1\}}\,g_2+\nchi_{\{u_1=u_2\}}\,(g_1\wedge g_2)
\in D[u_1\wedge u_2]
\end{equation}
whenever $u_1,u_2\in{\rm S}^p(\X)$, $g_1\in D[u_1]$ and $g_2\in D[u_2]$.
\item[$\rm ii)$] We say that $D$ is \emph{strongly local in the sense
of Shanmugalingam} provided
\begin{equation}
\nchi_B\,g_1+\nchi_{\X\setminus B}\,g_2\in D[u_2]
\quad\text{ for every }g_1\in D[u_1]\text{ and }g_2\in D[u_2]
\end{equation}
whenever $u_1,u_2\in{\rm S}^p(\X)$ satisfy $u_1=u_2$ $\mm$-a.e.\ on some Borel
set $B\subseteq\X$.
\end{itemize}
\end{definition}
\medskip

The above two notions of strong locality have been proposed in \cite{Timosh}
and \cite{Shanm}, respectively.
We now prove that they are actually both equivalent to our pointwise locality property:
\begin{lemma}
Let $(\X,\sfd,\mm)$ be a metric measure space and $p\in(1,\infty)$.
Fix any $D$-structure on $(\X,\sfd,\mm)$. Then the following are equivalent:
\begin{itemize}
\item[$\rm i)$] $D$ is pointwise local.
\item[$\rm ii)$] $D$ is strongly local in the sense of Shanmugalingam.
\item[$\rm iii)$] $D$ is strongly local in the sense of Timoshin.
\end{itemize}
\end{lemma}
\begin{proof}
\\
{\color{blue}${\rm i)}\Longrightarrow{\rm ii)}$} Fix $u_1,u_2\in{\rm S}^p(\X)$
such that $u_1=u_2$ $\mm$-a.e.\ on some $E\subseteq\X$ Borel.
Pick $g_1\in D[u_1]$ and $g_2\in D[u_2]$. Observe that
$\underline D(u_2-u_1)+g_1\in D\big[(u_2-u_1)+u_1\big]=D[u_2]$ by \textbf{A2},
so that we have $\big(\underline D(u_2-u_1)+g_1\big)\wedge g_2\in D[u_2]$ by \textbf{L4}.
Since $\underline D(u_2-u_1)=0$ $\mm$-a.e.\ on $B$ by \textbf{L2}, we see that
$\nchi_B\,g_1+\nchi_{\X\setminus B}\,g_2\geq\big(\underline D(u_2-u_1)+g_1\big)\wedge g_2$
holds $\mm$-a.e.\ in $\X$, whence accordingly we conclude that
$\nchi_B\,g_1+\nchi_{\X\setminus B}\,g_2\in D[u_2]$ by \textbf{A2}.
This shows the validity of ii).\\
{\color{blue}${\rm ii)}\Longrightarrow{\rm i)}$} First of all, let us prove \textbf{L1}.
Let $u\in{\rm S}^p(\X)$ and $c\in\R$ satisfy $u=c$ $\mm$-a.e.\ on some Borel
set $B\subseteq\X$. Given any $g\in D[u]$, we deduce from ii) that
$\nchi_{\X\setminus B}\,g\in D[u]$, thus accordingly
$\mathcal E_p(u|B)\leq\int_B(\nchi_{\X\setminus B}\,g)^p\,\d\mm=0$. This proves
the property \textbf{L1}.

To show property \textbf{L4}, fix $u\in{\rm S}^p(\X)$ and $g_1,g_2\in D[u]$.
Let us denote $B:=\{g_1\leq g_2\}$. Therefore ii) grants that
$g_1\wedge g_2=\nchi_B\,g_1+\nchi_{\X\setminus B}\,g_2\in D[u]$, thus
obtaining \textbf{L4}. By recalling Proposition \ref{prop:implications_locality},
we conclude that $D$ is pointwise local.\\
{\color{blue}${\rm i)}+{\rm ii)}\Longrightarrow{\rm iii)}$} Fix $u_1,u_2\in{\rm S}^p(\X)$,
$g_1\in D[u_1]$ and $g_2\in D[u_2]$.
Recall that $g_1\vee g_2\in D[u_1\wedge u_2]$ by axiom \textbf{A4}.
Hence by using property ii) twice we obtain that
\begin{equation}\label{eq:equiv_ptwse_local_aux}\begin{split}
\nchi_{\{u_1\leq u_2\}}\,g_1+\nchi_{\{u_1>u_2\}}\,(g_1\vee g_2)&\in D[u_1\wedge u_2],\\
\nchi_{\{u_2\leq u_1\}}\,g_2+\nchi_{\{u_2>u_1\}}\,(g_1\vee g_2)&\in D[u_1\wedge u_2].
\end{split}\end{equation}
The pointwise minimum between the two functions that are written in
\eqref{eq:equiv_ptwse_local_aux} -- namely given by
$\nchi_{\{u_1<u_2\}}\,g_1+\nchi_{\{u_2<u_1\}}\,g_2+\nchi_{\{u_1=u_2\}}\,(g_1\wedge g_2)$
-- belongs to the class $D[u_1\wedge u_2]$ as well by property \textbf{L4},
thus showing iii).\\
{\color{blue}${\rm iii)}\Longrightarrow{\rm i)}$} First of all, let us prove \textbf{L1}.
Fix a function $u\in{\rm S}^p(\X)$ that is $\mm$-a.e.\ equal to some constant $c\in\R$
on a Borel set $B\subseteq\X$. By using iii) and the fact that $0\in D[0]$,
we have that
\begin{equation}\label{eq:equiv_ptwse_local_aux2}\begin{split}
\nchi_{\{u<c\}}\,g&\in D\big[(u-c)\wedge 0\big]=D\big[-(u-c)^+\big]
=D\big[(u-c)^+\big],\\
\nchi_{\{u>c\}}\,g&\in D\big[(c-u)\wedge 0\big]=D\big[-(c-u)^+\big]
=D\big[(c-u)^+\big].
\end{split}\end{equation}
Since $u-c=(u-c)^+ -(c-u)^+$, we know from \textbf{A2} and
\eqref{eq:equiv_ptwse_local_aux2} that
\[\nchi_{\{u\neq c\}}\,g=\nchi_{\{u<c\}}\,g+\nchi_{\{u>c\}}\,g\in D[u-c]=D[u],\]
whence $\mathcal E_p(u|B)\leq\int_B(\nchi_{\{u\neq c\}}\,g)^p\,\d\mm=0$.
This proves the property \textbf{L1}.

To show property \textbf{L4}, fix $u\in{\rm S}^p(\X)$ and $g_1,g_2\in D[u]$.
Hence \eqref{eq:loc_Timosh} with $u_1=u_2:=u$ simply reads as $g_1\wedge g_2\in D[u]$,
which gives \textbf{L4}. This proves that $D$ is pointwise local.
\end{proof}
\begin{remark}[\textbf{L1} does not imply \textbf{L2}]{\rm
In general, as we are going to show in the following example, it can happen that
a $D$-structure satisfies \textbf{L1} but not \textbf{L2}.

Let $G=(V,E)$ be a locally finite connected graph. The distance $\sfd(x,y)$ between two
vertices $x,y\in V$ is defined as the minimum length of a path joining $x$ to $y$,
while as a reference measure $\mm$ on $V$ we choose the counting measure. Notice that
any function $u:\,V\to\R$ is locally Lipschitz and that any bounded subset of $V$
is finite. We define a $D$-structure on the metric measure space $(V,\sfd,\mm)$ in the following way:
\begin{equation}\label{eq:D-structure_graphs}
D[u]:=\Big\{g:\,V\to [0,+\infty]\;\Big|\;\big|u(x)-u(y)\big|\leq g(x)+g(y)
\text{ for any }x,y\in V\text{ with }x\sim y\Big\}
\end{equation}
for every $u:\,V\to\R$, where the notation $x\sim y$ indicates that $x$ and $y$ are adjacent vertices,
i.e.\ that there exists an edge in $E$ joining $x$ to $y$.

We claim that $D$ fulfills \textbf{L1}. To prove it, suppose that some function $u:\,\X\to\R$ is constant on some
set $B\subseteq V$, say $u(x)=c$ for every $x\in B$. Define the function $g:\,V\to[0,+\infty)$ as
$$g(x):=\left\{\begin{array}{ll}
0\\
|c|+\big|u(x)\big|
\end{array}\quad\begin{array}{ll}
\text{ if }x\in B,\\
\text{ if }x\in V\setminus B.
\end{array}\right.$$
Hence $g\in D[u]$ and $\int_B g^p\,\d\mm=0$, so that $\mathcal{E}_p(u|B)=0$.
This proves the validity of \textbf{L1}.

On the other hand, if $V$ contains more than one vertex, then \textbf{L2} is not satisfied.
Indeed, consider any non-constant function $u:\,V\to\R$. Clearly any
pseudo-gradient $g\in D[u]$ of $u$ is not identically zero, thus there exists
$x\in V$ such that $\underline D u(x)>0$. Since $u$ is trivially constant
on the set $\{x\}$, we then conclude that property \textbf{L2} does not hold.
\fr}\end{remark}

Hereafter, we shall focus our attention on the pointwise local $D$-structures.
Under these locality assumptions, one can show the following calculus
rules for minimal pseudo-gradients, whose proof is suitably adapted from
analogous results that have been proved in \cite{AmbrosioGigliSavare11}.
\begin{proposition}[Calculus rules for $\underline D u$]\label{prop:properties_Du}
Let $(\X,\sfd,\mm)$ be a metric measure space and let $p\in(1,\infty)$.
Consider a pointwise local $D$-structure on $(\X,\sfd,\mm)$.
Then the following hold:
\begin{itemize}
\item[$\rm i)$] Let $u\in{\rm S}^p(\X)$ and let $N\subseteq\R$ be a Borel
set with $\mathcal L^1(N)=0$. Then the equality $\underline D u=0$ holds $\mm$-a.e.\ in
$u^{-1}(N)$.
\item[$\rm ii)$] \textsc{Chain rule}. Let $u\in{\rm S}^p(\X)$ and
$\varphi\in{\rm LIP}(\R)$. Then $|\varphi'|\circ u\,\underline D u\in D[\varphi\circ u]$.
More precisely, $\varphi\circ u\in{\rm S}^p(\X)$ and
$\underline D(\varphi\circ u)=|\varphi'|\circ u\,\underline D u$ holds $\mm$-a.e.\ in $\X$.
\item[$\rm iii)$] \textsc{Leibniz rule}. Let $u,v\in{\rm S}^p(\X)\cap L^\infty(\mm)$.
Then $|u|\,\underline D v+|v|\,\underline D u\in D[uv]$. In other words,
$uv\in{\rm S}^p(\X)\cap L^\infty(\mm)$ and $\underline D(uv)\leq|u|\,\underline D v+|v|\,\underline D u$
holds $\mm$-a.e.\ in $\X$.
\end{itemize}
\end{proposition}
\begin{proof}\\
{\color{blue}\textsc{Step 1.}} First, consider $\varphi$ affine,
say $\varphi(t)=\alpha\,t+\beta$. Then
$|\varphi'|\circ u\,\underline D u=|\alpha|\,\underline D u\in D[\varphi\circ u]$
by Remark \ref{rmk:conseq_D-structure} and \textbf{A2}.
Now suppose that the function $\varphi$ is piecewise affine, i.e.\ there exists
a sequence $(a_k)_{k\in\Z}\subseteq\R$, with $a_k<a_{k+1}$
for all $k\in\Z$ and $a_0=0$, such that each $\varphi\restr{[a_k,a_{k+1}]}$ is an affine function.
Let us denote $A_k:=u^{-1}\big([a_k,a_{k+1})\big)$ and $u_k:=(u\vee a_k)\wedge a_{k+1}$ for every index $k\in\Z$.
By combining \textbf{L3} with the axioms \textbf{A2} and \textbf{A5}, we can see that
$\nchi_{A_k}\,\underline D u\in D[u_k]$ for every $k\in\Z$. Called $\varphi_k:\,\R\to\R$
that affine function coinciding with $\varphi$ on $[a_k,a_{k+1})$, we deduce from
the previous case that
$|\varphi'_k|\circ u_k\,\underline D u_k\in D[\varphi_k\circ u_k]=D[\varphi\circ u_k]$,
whence we have that $|\varphi'|\circ u_k\,\nchi_{A_k}\,\underline D u\in D[\varphi\circ u_k]$
by \textbf{L5}, \textbf{A2} and \textbf{L2}.
Let us define $(v_n)_n\subseteq{\rm S}^p(\X)$ as
\[v_n:=\varphi(0)+\sum_{k=0}^n\big(\varphi\circ u_k-\varphi(a_k)\big)
+\sum_{k=-n}^{-1}\big(\varphi\circ u_k-\varphi(a_{k+1})\big)
\quad\text{ for every }n\in\N.\]
Hence $g_n:=\sum_{k=-n}^n|\varphi'|\circ u_k\,\nchi_{A_k}\,\underline D u\in D[v_n]$
for all $n\in\N$ by \textbf{A2} and Remark \ref{rmk:conseq_D-structure}.
Given that one has $v_n\to\varphi\circ u$ in $L^p_{loc}(\mm)$ and $g_n\to|\varphi'|\circ u\,\underline D u$ in $L^p(\mm)$ as $n\to\infty$, we finally conclude that
$|\varphi'|\circ u\,\underline D u\in D[\varphi\circ u]$, as required.\\
{\color{blue}\textsc{Step 2.}} We aim to prove the chain rule for $\varphi\in C^1(\R)\cap{\rm LIP}(\R)$. For any $n\in\N$, let us denote by $\varphi_n$ the piecewise
affine function interpolating the points $\big(k/2^n,\varphi(k/2^n)\big)$
with $k\in\Z$. We call $D\subseteq\R$ the countable set
$\big\{k/2^n\,:\,k\in\Z,\,n\in\N\big\}$. Therefore $\varphi_n$ uniformly converges to
$\varphi$ and $\varphi'_n(t)\to\varphi'(t)$ for all $t\in\R\setminus D$.
In particular, the functions $g_n:=|\varphi'_n|\circ u\,\underline D u$
converge $\mm$-a.e.\ to $|\varphi'|\circ u\,\underline D u$ by \textbf{L2}.
Moreover, ${\rm Lip}(\varphi_n)\leq{\rm Lip}(\varphi)$ for every $n\in\N$ by construction,
so that $(g_n)_n$ is a bounded sequence in $L^p(\mm)$. This implies that
(up to a not relabeled subsequence) $g_n\weakto|\varphi'|\circ u\,\underline D u$
weakly in $L^p(\mm)$. Now apply Mazur lemma: for any $n\in\N$, there exists
$(\alpha^n_i)_{i=n}^{N_n}\subseteq[0,1]$ such that $\sum_{i=n}^{N_n}\alpha^n_i=1$
and $h_n:=\sum_{i=n}^{N_n}\alpha^n_i\,g_i\overset{n}\to|\varphi'|\circ u\,\underline D u$
strongly in $L^p(\mm)$.
Given that $g_n\in D[\varphi_n\circ u]$ for every $n\in\N$ by
\textsc{Step 1}, we deduce from axiom \textbf{A2} that $h_n\in D[\psi_n\circ u]$
for every $n\in\N$, where $\psi_n:=\sum_{i=n}^{N_n}\alpha^n_i\,\varphi_i$.
Finally, it clearly holds that $\psi_n\circ u\to\varphi\circ u$ in $L^p_{loc}(\mm)$,
whence $|\varphi'|\circ u\,\underline D u\in D[\varphi\circ u]$ by \textbf{A5}.\\
{\color{blue}\textsc{Step 3.}} We claim that
\begin{equation}\label{eq:properties_Du_claim}
\underline D u=0\;\;\;\mm\text{-a.e.\ in }u^{-1}(K),
\quad\text{ for every }K\subseteq\R\text{ compact with }\mathcal L^1(K)=0.
\end{equation}
For any $n\in\N\setminus\{0\}$, define $\psi_n:=n\,\sfd(\cdot,K)\wedge 1$
and denote by $\varphi_n$ the primitive of $\psi_n$ such that $\varphi_n(0)=0$.
Since each $\psi_n$ is continuous and bounded, any function $\varphi_n$ is
of class $C^1$ and Lipschitz. By applying the dominated convergence theorem
we see that the $\mathcal L^1$-measure of the $\eps$-neighbourhood of $K$
converges to $0$ as $\eps\searrow 0$, thus accordingly $\varphi_n$ uniformly converges
to ${\rm id}_\R$ as $n\to\infty$. This implies that $\varphi_n\circ u\to u$
in $L^p_{loc}(\mm)$. Moreover, we know from \textsc{Step 2} that
$|\psi_n|\circ u\,\underline D u\in D[\varphi_n\circ u]$, thus also
$\nchi_{\X\setminus u^{-1}(K)}\,\underline D u\in D[\varphi_n\circ u]$.
Hence $\nchi_{\X\setminus u^{-1}(K)}\,\underline D u\in D[u]$ by \textbf{A5},
which forces the equality $\underline D u=0$ to hold $\mm$-a.e.\ in $u^{-1}(K)$,
proving \eqref{eq:properties_Du_claim}.\\
{\color{blue}\textsc{Step 4.}} We are in a position to prove i). Choose any
$\mm'\in\mathscr P(\X)$ such that $\mm\ll\mm'\ll\mm$ and call $\mu:=u_*\mm'$.
Then $\mu$ is a Radon measure on $\R$, in particular it is inner regular.
We can thus find an increasing sequence of compact sets $K_n\subseteq N$
such that $\mu\big(N\setminus\bigcup_n K_n\big)=0$. We already know from
\textsc{Step 3} that $\underline D u=0$ holds $\mm$-a.e.\ in $\bigcup_n u^{-1}(K_n)$.
Since $u^{-1}(N)\setminus\bigcup_n u^{-1}(K_n)$ is $\mm$-negligible
by definition of $\mu$, we conclude that $\underline D u=0$
holds $\mm$-a.e.\ in $u^{-1}(N)$. This shows the validity of property i).\\
{\color{blue}\textsc{Step 5.}} We now prove ii). Let us fix $\varphi\in{\rm LIP}(\R)$.
Choose some convolution kernels $(\rho_n)_n$ and define $\varphi_n:=\varphi*\rho_n$
for all $n\in\N$. Then $\varphi_n\to\varphi$ uniformly and $\varphi'_n\to\varphi'$
pointwise $\mathcal L^1$-a.e., whence accordingly $\varphi_n\circ u\to\varphi\circ u$ in
$L^p_{loc}(\mm)$ and $|\varphi'_n|\circ u\,\underline D u\to|\varphi'|\circ u\,\underline D u$
pointwise $\mm$-a.e.\ in $\X$. Since $|\varphi'_n|\circ u\,\underline D u
\leq{\rm Lip}(\varphi)\,\underline D u$ for all $n\in\N$, there exists a (not relabeled)
subsequence such that $|\varphi'_n|\circ u\,\underline D u\weakto|\varphi'|\circ u\,\underline D u$
weakly in $L^p(\mm)$. We know that $|\varphi'_n|\circ u\,\underline D u\in D[\varphi_n\circ u]$ for all $n\in\N$ because the chain rule holds for all $\varphi_n\in C^1(\R)\cap{\rm LIP}(\R)$,
hence by combining Mazur lemma and \textbf{A5} as in \textsc{Step 2} we obtain that
$|\varphi'|\circ u\,\underline D u\in D[\varphi\circ u]$, so that
$\varphi\circ u\in{\rm S}^p(\X)$ and the inequality
$\underline D(\varphi\circ u)\leq|\varphi'|\circ u\,\underline D u$
holds $\mm$-a.e.\ in $\X$.\\
{\color{blue}\textsc{Step 6.}} We conclude the proof of ii) by showing that one
actually has $\underline D(\varphi\circ u)=|\varphi'|\circ u\,\underline D u$.
We can suppose without loss of generality that ${\rm Lip}(\varphi)=1$.
Let us define the functions $\psi_\pm$ as $\psi_\pm(t):=\pm t-\varphi(t)$
for all $t\in\R$. Then it holds $\mm$-a.e.\ in $u^{-1}\big(\{\pm\varphi'\geq 0\}\big)$ that
$$\underline D u=\underline D(\pm u)\leq
\underline D(\varphi\circ u)+\underline D(\psi_\pm\circ u)
\leq\big(|\varphi'|\circ u+|\psi'_\pm|\circ u\big)\,\underline D u=\underline D u,$$
which forces the equality $\underline D(\varphi\circ u)=\pm\varphi'\circ u\,\underline D u$
to hold $\mm$-a.e.\ in the set $u^{-1}\big(\{\pm\varphi'\geq 0\}\big)$.
This grants the validity of $\underline D(\varphi\circ u)=|\varphi'|\circ u\,\underline D u$,
thus completing the proof of item ii).\\
{\color{blue}\textsc{Step 7.}} We show iii) for the case in which
$u,v\geq c$ is satisfied $\mm$-a.e.\ in $\X$, for some $c>0$.
Call $\eps:=\min\{c,c^2\}$ and note that the function $\log$ is Lipschitz on
the interval $[\eps,+\infty)$, then choose any Lipschitz function $\varphi:\,\R\to\R$
that coincides with $\log$ on $[\eps,+\infty)$.
Now call $C$ the constant $\log\big({\|uv\|}_{L^\infty(\mm)}\big)$
and choose a Lipschitz function $\psi:\,\R\to\R$ such that $\psi=\exp$ on the interval
$[\log\eps,C]$. By applying twice the chain rule ii), we thus deduce that
$uv\in{\rm S}^p(\X)$ and the $\mm$-a.e.\ inequalities
\[\begin{split}
\underline D(uv)&\leq|\psi'|\circ\varphi\circ(uv)\,\underline D
\big(\varphi\circ(uv)\big)\leq|uv|\,\big(\underline D\log u+\underline D\log v\big)\\
&=|uv|\left(\frac{\underline D u}{|u|}+\frac{\underline D v}{|v|}\right)=
|u|\,\underline D v+|v|\,\underline D u.
\end{split}\]
Therefore the Leibniz rule iii) is verified under the additional assumption that $u,v\geq c>0$.\\
{\color{blue}\textsc{Step 8.}}
We conclude by proving item iii) for general $u,v\in{\rm S}^p(\X)\cap L^\infty(\mm)$.
Given any $n\in\N$ and $k\in\Z$, let us denote $I_{n,k}:=\big[k/n,(k+1)/n\big)$.
Call $\varphi_{n,k}:\,\R\to\R$ the continuous function that is the identity on $I_{n,k}$ and constant elsewhere.
For any $n\in\N$, let us define
\begin{align*}
u_{n,k}&:=u-\frac{k-1}{n},&\tilde u_{n,k}&:=\varphi_{n,k}\circ u-\frac{k-1}{n}
&\text{ for all }k\in\Z,\\
v_{n,\ell}&:=v-\frac{\ell-1}{n},&\tilde v_{n,\ell}
&:=\varphi_{n,\ell}\circ v-\frac{\ell-1}{n}&\text{ for all }\ell\in\Z.
\end{align*}
Notice that the equalities $u_{n,k}=\tilde u_{n,k}$ and $v_{n,\ell}=\tilde v_{n,\ell}$
hold $\mm$-a.e.\ in $u^{-1}(I_{n,k})$ and $v^{-1}(I_{n,\ell})$, respectively. Hence
$\underline D u_{n,k}=\underline D\tilde u_{n,k}=\underline D u$ and
$\underline D v_{n,\ell}=\underline D\tilde v_{n,\ell}=\underline D v$
hold $\mm$-a.e.\ in $u^{-1}(I_{n,k})$ and $v^{-1}(I_{n,\ell})$, respectively, but
we also have that
\[\underline D(u_{n,k}\,v_{n,\ell})=\underline D(\tilde u_{n,k}\,\tilde v_{n,\ell})
\quad\text{ is verified }\mm\text{-a.e.\ in }u^{-1}(I_{n,k})\cap v^{-1}(I_{n,\ell}).\]
Moreover, we have the $\mm$-a.e.\ inequalities
$1/n\leq\tilde u_{n,k},\tilde v_{n,\ell}\leq 2/n$ by construction. Therefore for
any $k,\ell\in\Z$ it holds $\mm$-a.e.\ in $u^{-1}(I_{n,k})\cap v^{-1}(I_{n,\ell})$ that
\[\begin{split}
\underline D(uv)&\leq\underline D(\tilde u_{n,k}\,\tilde v_{n,\ell})
+\frac{|k-1|}{n}\,\underline D v_{n,\ell}+\frac{|\ell-1|}{n}\,\underline D u_{n,k}\\
&\leq|\tilde v_{n,\ell}|\,\underline D\tilde u_{n,k}+
|\tilde u_{n,k}|\,\underline D\tilde v_{n,\ell}+
\frac{|k-1|}{n}\,\underline D v_{n,\ell}+\frac{|\ell-1|}{n}\,\underline D u_{n,k}\\
&\leq\left(|v|+\frac{4}{n}\right)\underline D u
+\left(|u|+\frac{4}{n}\right)\underline D v,
\end{split}\]
where the second inequality follows from the case $u,v\geq c>0$, treated in \textsc{Step 7}.
This implies that the inequality
$\underline D(uv)\leq|u|\,\underline D v+|v|\,\underline D u+4\,(\underline D u
+\underline D v)/n$ holds $\mm$-a.e.\ in $\X$.
Given that $n\in\N$ is arbitrary, the Leibniz rule iii) follows.
\end{proof}
\section{Cotangent module associated to a \texorpdfstring{$D$}{D}-structure}
It is shown in \cite{Gigli14} that any metric measure space possesses a
first-order differential structure, whose construction relies upon the
notion of \emph{$L^p(\mm)$-normed $L^\infty(\mm)$-module}. For completeness,
we briefly recall its definition and we refer to \cite{Gigli14,Gigli17} for
a comprehensive exposition of this topic.
\begin{definition}[Normed module]
Let $(\X,\sfd,\mm)$ be a metric measure space and $p\in[1,\infty)$.
Then an \emph{$L^p(\mm)$-normed $L^\infty(\mm)$-module} is any quadruplet
$\big(\mathscr M,{\|\cdot\|}_{\mathscr M},\,\cdot\,,|\cdot|\big)$ such that
\begin{itemize}
\item[$\rm i)$] $\big(\mathscr M,{\|\cdot\|}_{\mathscr M}\big)$ is a Banach space,
\item[$\rm ii)$] $(\mathscr M,\cdot)$ is an algebraic module over the
commutative ring $L^\infty(\mm)$,
\item[$\rm iii)$] the \emph{pointwise norm} operator
$|\cdot|:\,\mathscr M\to L^p(\mm)^+$ satisfies
\begin{equation}\label{eq:ptwse_norm}\begin{split}
|f\cdot v|=|f||v|\;\;\;\mm\text{-a.e.}&\quad\text{ for every }f\in L^\infty(\mm)\text{ and }v\in\mathscr M,\\
{\|v\|}_{\mathscr M}={\big\||v|\big\|}_{L^p(\mm)}&\quad\text{ for every }v\in\mathscr M.
\end{split}\end{equation}
\end{itemize}
\end{definition}

A key role in \cite{Gigli14} is played by the \emph{cotangent module} $L^2(T^*\X)$,
which has a structure of $L^2(\mm)$-normed $L^\infty(\mm)$-module;
see \cite[Theorem/Definition 1.8]{Gigli17} for its characterisation.
The following result shows that a generalised version of such object can be
actually associated to any $D$-structure,
provided the latter is assumed to be pointwise local.
\begin{theorem}[Cotangent module associated to a $D$-structure]\label{thm:cot_mod}
Let $(\X,\sfd,\mm)$ be any metric measure space and let $p\in(1,\infty)$.
Consider a pointwise local $D$-structure on $(\X,\sfd,\mm)$.
Then there exists a unique couple $\big(L^p(T^*\X;D),\d\big)$, where $L^p(T^*\X;D)$
is an $L^p(\mm)$-normed $L^\infty(\mm)$-module and $\d:\,{\rm S}^p(\X)\to L^p(T^*\X;D)$
is a linear map, such that the following hold:
\begin{itemize}
\item[$\rm i)$] the equality $|\d u|=\underline{D}u$ is satisfied $\mm$-a.e.\ in
$\X$ for every $u\in{\rm S}^p(\X)$,
\item[$\rm ii)$] the vector space $\mathcal{V}$ of all elements of the form $\sum_{i=1}^n\nchi_{B_i}\,\d u_i$,
where $(B_i)_i$ is a Borel partition of $\X$ and $(u_i)_i\subseteq{\rm S}^p(\X)$, is dense in the space $L^p(T^*\X;D)$.
\end{itemize}
Uniqueness has to be intended up to unique isomorphism: given another such couple $(\mathscr{M},\d')$,
there is a unique isomorphism $\Phi:\,L^p(T^*\X;D)\to\mathscr{M}$
such that $\Phi(\d u)=\d' u$ for all $u\in S^p(\X)$.

The space $L^p(T^*\X;D)$ is called \emph{cotangent module},
while the map $\d$ is called \emph{differential}.
\end{theorem}
\begin{proof}\\
{\color{blue}\textsc{Uniqueness.}}
Consider any element $\omega\in\mathcal{V}$ written as
$\omega=\sum_{i=1}^n\nchi_{B_i}\,\d u_i$,
with $(B_i)_i$ Borel partition of $\X$ and $u_1,\ldots,u_n\in S^p(\X)$. Notice that the requirements
that $\Phi$ is $L^\infty(\mm)$-linear and $\Phi\circ\d=\d'$
force the definition $\Phi(\omega):=\sum_{i=1}^n\nchi_{B_i}\,\d'u_i$. The $\mm$-a.e.\ equality
$$\big|\Phi(\omega)\big|=\sum_{i=1}\nchi_{B_i}\,|\d' u_i|=
\sum_{i=1}^n\nchi_{B_i}\,\underline{D}u_i=\sum_{i=1}^n\nchi_{B_i}\,|\d u_i|=|\omega|$$
grants that $\Phi(\omega)$ is well-defined, in the sense that it does not depend on the particular way of
representing $\omega$, and that $\Phi:\,\mathcal{V}\to\mathscr{M}$ preserves the pointwise norm.
In particular, one has that the map $\Phi:\,\mathcal{V}\to\mathscr{M}$ is (linear and) continuous.
Since $\mathcal{V}$ is dense in $L^p(T^*\X;D)$, we can uniquely extend $\Phi$
to a linear and continuous map $\Phi:\,L^p(T^*\X;D)\to\mathscr{M}$, which also preserves the pointwise
norm. Moreover, we deduce from the very definition of $\Phi$ that the identity $\Phi(h\,\omega)=h\,\Phi(\omega)$
holds for every $\omega\in\mathcal{V}$ and $h\in{\sf Sf}(\X)$, whence the $L^\infty(\mm)$-linearity
of $\Phi$ follows by an approximation argument. Finally, the image $\Phi(\mathcal{V})$ is dense
in $\mathscr{M}$, which implies that $\Phi$ is surjective. Therefore $\Phi$ is the unique isomorphism
satisfying $\Phi\circ\d=\d'$.\\
{\color{blue}\textsc{Existence.}}
First of all, let us define the \emph{pre-cotangent module} as
\[{\sf Pcm}:=\left\{\big\{(B_i,u_i)\big\}_{i=1}^n\;\bigg|\begin{array}{ll}
\;n\in\N,\;u_1,\ldots,u_n\in{\rm S}^p(\X),\\
(B_i)_{i=1}^n\text{ Borel partition of }\X
\end{array}\right\}.\]
We define an equivalence relation on $\sf Pcm$ as follows: we declare that
$\big\{(B_i,u_i)\big\}_i\sim\big\{(C_j,v_j)\big\}_j$ provided $\underline{D}(u_i-v_j)=0$
holds $\mm$-a.e.\ on $B_i\cap C_j$ for every $i,j$. The equivalence class of
an element $\big\{(B_i,u_i)\big\}_i$ of $\sf Pcm$ will be denoted by $[B_i,u_i]_i$.
We can endow the quotient ${\sf Pcm}/\sim$ with a vector space structure:
\begin{equation}\label{eq:def_vector_sp_operations}\begin{split}
[B_i,u_i]_i+[C_j,v_j]_j&:=[B_i\cap C_j,u_i+v_j]_{i,j},\\
\lambda\,[B_i,u_i]_i&:=[B_i,\lambda\,u_i]_i,
\end{split}\end{equation}
for every $[B_i,u_i]_i,[C_j,v_j]_j\in{\sf Pcm}/\sim$ and $\lambda\in\R$.
We only check that the sum operator is well-defined; the proof of the well-posedness of
the multiplication by scalars follows along the same lines.
Suppose that $\big\{(B_i,u_i)\big\}_i\sim\big\{(B'_k,u'_k)\big\}_k$
and $\big\{(C_j,v_j)\big\}_j\sim\big\{(C'_\ell,v'_\ell)\big\}_\ell$,
in other words $\underline D(u_i-u'_k)=0$ $\mm$-a.e.\ on $B_i\cap B'_k$
and $\underline D(v_j-v'_\ell)=0$ $\mm$-a.e.\ on $C_j\cap C'_\ell$ for every
$i,j,k,\ell$, whence accordingly
\[\underline D\big((u_i+v_j)-(u'_k+v'_\ell)\big)\overset{\textbf{L5}}\leq
\underline D(u_i-u'_k)+\underline D(v_j-v'_\ell)=0\quad\text{ holds }
\mm\text{-a.e.\ on }(B_i\cap C_j)\cap(B'_k\cap C'_\ell).\]
This shows that $\big\{(B_i\cap C_j,u_i+v_j)\big\}_{i,j}\sim
\big\{(B'_k\cap C'_\ell,u'_k+v'_\ell)\big\}_{k,\ell}$, thus proving
that the sum operator defined in \eqref{eq:def_vector_sp_operations} is well-posed.
Now let us define
\begin{equation}\label{eq:def_norm}
{\big\|[B_i,u_i]_i\big\|}_{L^p(T^*\X;D)}:=
\sum_{i=1}^n\bigg(\int_{B_i}(\underline{D}u_i)^p\,\d\mm\bigg)^{1/p}
\quad\text{ for every }[B_i,u_i]_i\in{\sf Pcm}/\sim.
\end{equation}
Such definition is well-posed: if $\big\{(B_i,u_i)\big\}_i\sim\big\{(C_j,v_j)\big\}_j$
then for all $i,j$ it holds that
\[|\underline D u_i-\underline D v_j|\overset{\textbf{L5}}\leq
\underline D(u_i-v_j)=0\quad\mm\text{-a.e.\ on }B_i\cap C_j,\]
i.e.\ that the equality $\underline D u_i=\underline D v_j$ is satisfied
$\mm$-a.e.\ on $B_i\cap C_j$. Therefore one has that
\[\begin{split}
\sum_i\bigg(\int_{B_i}(\underline D u_i)^p\,\d\mm\bigg)^{1/p}
&=\sum_{i,j}\bigg(\int_{B_i\cap C_j}(\underline D u_i)^p\,\d\mm\bigg)^{1/p}
=\sum_{i,j}\bigg(\int_{B_i\cap C_j}(\underline D v_j)^p\,\d\mm\bigg)^{1/p}\\
&=\sum_j\bigg(\int_{C_j}(\underline D v_j)^p\,\d\mm\bigg)^{1/p},
\end{split}\]
which grants that ${\|\cdot\|}_{L^p(T^*\X;D)}$ in \eqref{eq:def_norm} is well-defined.
The fact that it is a norm on ${\sf Pcm}/\sim$ easily follows from standard verifications.
Hence let us define
\[\begin{split}
&L^p(T^*\X;D):=\text{completion of }\big({\sf Pcm}/\sim,{\|\cdot\|}_{L^p(T^*\X;D)}\big),\\
&\d:\,{\rm S}^p(\X)\to L^p(T^*\X;D),\;\;\;\d u:=[\X,u]\text{ for every }u\in{\rm S}^p(\X).
\end{split}\]
Observe that $L^p(T^*\X;D)$ is a Banach space and that $\d$ is a linear operator.
Furthermore, given any $[B_i,u_i]_i\in{\sf Pcm}/\sim$ and
$h=\sum_j\lambda_j\,\nchi_{C_j}\in{\sf Sf}(\X)$,
where $(\lambda_j)_j\subseteq\R$ and $(C_j)_j$ is a Borel partition of $\X$, we set
\[\begin{split}
\big|[B_i,u_i]_i\big|&:=\sum_i\nchi_{B_i}\,\underline{D}u_i,\\
h\,[B_i,u_i]_i&:=[B_i\cap C_j,\lambda_j\,u_i]_{i,j}.
\end{split}\]
One can readily prove that such operations, which are well-posed again by the pointwise
locality of $D$, can be uniquely extended to a pointwise norm
$|\cdot|:\,L^p(T^*\X;D)\to L^p(\mm)^+$ and to a multiplication by $L^\infty$-functions
$L^\infty(\mm)\times L^p(T^*\X;D)\to L^p(T^*\X;D)$, respectively.
Therefore the space $L^p(T^*\X;D)$ turns out to be an $L^p(\mm)$-normed $L^\infty(\mm)$-module
when equipped with the operations described so far. In order to conclude,
it suffices to notice that
$$|\d u|=\big|[\X,u]\big|=\underline{D}u\;\;\;\text{holds }\mm\text{-a.e.}
\quad\text{ for every }u\in{\rm S}^p(\X)$$
and that $[B_i,u_i]_i=\sum_i\nchi_{B_i}\,\d u_i$ for all $[B_i,u_i]_i\in{\sf Pcm}/\sim$,
giving i) and ii), respectively.
\end{proof}
\bigskip

In full analogy with the properties of the cotangent module that is
studied in \cite{Gigli14}, we can show that the differential $\d$ introduced
in Theorem \ref{thm:cot_mod} is a closed operator, which satisfies both
the chain rule and the Leibniz rule.
\begin{theorem}[Closure of the differential]\label{thm:closure_d}
Let $(\X,\sfd,\mm)$ be a metric measure space and let $p\in(1,\infty)$.
Consider a pointwise local $D$-structure on $(\X,\sfd,\mm)$.
Then the differential operator $\d$ is \emph{closed}, i.e.\ if a sequence $(u_n)_n\subseteq{\rm S}^p(\X)$
converges in $L^p_{loc}(\mm)$ to some $u\in L^p_{loc}(\mm)$ and $\d u_n\weakto\omega$
weakly in $L^p(T^*\X;D)$ for some $\omega\in L^p(T^*\X;D)$, then $u\in{\rm S}^p(\X)$
and $\d u=\omega$.
\end{theorem}
\begin{proof}
Since $\d$ is linear, we can assume with no loss of generality that
$\d u_n\to\omega$ in $L^p(T^*\X;D)$ by Mazur lemma, so that
$\d(u_n-u_m)\to\omega-\d u_m$ in $L^p(T^*\X;D)$ for any $m\in\N$.
In particular, one has $u_n-u_m\to u-u_m$ in $L^p_{loc}(\mm)$ and
$\underline D(u_n-u_m)=\big|\d(u_n-u_m)\big|\to|\omega-\d u_m|$
in $L^p(\mm)$ as $n\to\infty$ for all $m\in\N$,
whence $u-u_m\in{\rm S}^p(\X)$ and $\underline D(u-u_m)\leq|\omega-\d u_m|$
holds $\mm$-a.e.\ for all $m\in\N$ by \textbf{A5} and \textbf{L5}. Therefore
$u=(u-u_0)+u_0\in{\rm S}^p(\X)$ and
\[\begin{split}
\lims_{m\to\infty}{\|\d u-\d u_m\|}_{L^p(T^*\X;D)}
&=\lims_{m\to\infty}{\big\|\underline D(u-u_m)\big\|}_{L^p(\mm)}
\leq\lims_{m\to\infty}{\|\omega-\d u_m\|}_{L^p(T^*\X;D)}\\
&=\lims_{m\to\infty}\lim_{n\to\infty}{\|\d u_n-\d u_m\|}_{L^p(T^*\X;D)}=0,
\end{split}\]
which grants that $\d u_m\to\d u$ in $L^p(T^*\X;D)$ as $m\to\infty$ and
accordingly that $\d u=\omega$.
\end{proof}
\begin{proposition}[Calculus rules for $\d u$]
Let $(\X,\sfd,\mm)$ be any metric measure space and let $p\in(1,\infty)$.
Consider a pointwise local $D$-structure on $(\X,\sfd,\mm)$.
Then the following hold:
\begin{itemize}
\item[$\rm i)$] Let $u\in{\rm S}^p(\X)$ and let $N\subseteq\R$ be a Borel
set with $\mathcal L^1(N)=0$. Then $\nchi_{u^{-1}(N)}\,\d u=0$.
\item[$\rm ii)$] \textsc{Chain rule}. Let $u\in{\rm S}^p(\X)$ and
$\varphi\in{\rm LIP}(\R)$ be given. Recall that $\varphi\circ u\in{\rm S}^p(\X)$
by Proposition \ref{prop:properties_Du}. Then $\d(\varphi\circ u)=\varphi'\circ u\,\d u$.
\item[$\rm iii)$] \textsc{Leibniz rule}. Let $u,v\in{\rm S}^p(\X)\cap L^\infty(\mm)$
be given. Recall that $uv\in{\rm S}^p(\X)\cap L^\infty(\mm)$
by Proposition \ref{prop:properties_Du}. Then $\d(uv)=u\,\d v+v\,\d u$.
\end{itemize}
\end{proposition}
\begin{proof}\\
{\color{blue}$\rm i)$} We have that $|\d u|=\underline D u=0$ holds $\mm$-a.e.\ on
$u^{-1}(N)$ by item i) of Proposition \ref{prop:properties_Du},
thus accordingly $\nchi_{u^{-1}(N)}\,\d u=0$, as required.\\
{\color{blue}$\rm ii)$} If $\varphi$ is an affine function, say $\varphi(t)=\alpha\,t+\beta$,
then $\d(\varphi\circ u)=\d(\alpha\,u+\beta)=\alpha\,\d u=\varphi'\circ u\,\d u$.
Now suppose that $\varphi$ is a piecewise affine function. Say that $(I_n)_n$ is a sequence
of intervals whose union covers the whole real line $\R$ and that $(\psi_n)_n$ is a sequence of
affine functions such that $\varphi\restr{I_n}=\psi_n$ holds for every $n\in\N$.
Since $\varphi'$ and $\psi'_n$ coincide $\mathcal L^1$-a.e.\ in the interior of $I_n$, we have that
$\d(\varphi\circ f)=\d(\psi_n\circ f)=\psi'_n\circ f\,\d f=\varphi'\circ f\,\d f$ holds $\mm$-a.e.\ on
$f^{-1}(I_n)$ for all $n$, so that $\d(\varphi\circ u)=\varphi'\circ u\,\d u$ is verified
$\mm$-a.e.\ on $\bigcup_n u^{-1}(I_n)=\X$.

To prove the case of a general Lipschitz function $\varphi:\,\R\to\R$, we want to approximate
$\varphi$ with a sequence of piecewise affine functions: for any $n\in\N$, let us denote
by $\varphi_n$ the function that coincides with $\varphi$ at $\big\{k/2^n\,:\,k\in\Z\big\}$
and that is affine on the interval $\big[k/2^n,(k+1)/2^n\big]$ for every $k\in\Z$.
It is clear that ${\rm Lip}(\varphi_n)\leq{\rm Lip}(\varphi)$ for all $n\in\N$.
Moreover, one can readily check that, up to a not relabeled subsequence,
$\varphi_n\to\varphi$ uniformly on $\R$ and $\varphi'_n\to\varphi'$ pointwise
$\mathcal L^1$-almost everywhere. The former grants that
$\varphi_n\circ u\to\varphi\circ u$ in $L^p_{loc}(\mm)$. Given that
$|\varphi'_n-\varphi'|^p\circ u\,(\underline D u)^p
\leq 2^p\,{\rm Lip}(\varphi)^p\,(\underline D u)^p\in L^1(\mm)$ for all $n\in\N$
and $|\varphi'_n-\varphi'|^p\circ u\,(\underline D u)^p\to 0$ pointwise
$\mm$-a.e.\ by the latter above together with i), we obtain
$\int|\varphi'_n-\varphi'|^p\circ u\,(\underline D u)^p\,\d\mm\to 0$ as $n\to\infty$
by the dominated convergence theorem. In other words,
$\varphi'_n\circ u\,\d u\to\varphi'\circ u\,\d u$ in the strong topology
of $L^p(T^*\X;D)$. Hence Theorem \ref{thm:closure_d} ensures that
$\d(\varphi\circ u)=\varphi'\circ u\,\d u$, thus proving the chain rule ii)
for any $\varphi\in{\rm LIP}(\R)$.\\
{\color{blue}$\rm iii)$} In the case $u,v\geq 1$, we argue as in
the proof of Proposition \ref{prop:properties_Du} to deduce from ii) that
\[\frac{\d(uv)}{uv}=\d\log(uv)=\d\big(\log(u)+\log(v)\big)=\d\log(u)+\d\log(v)=
\frac{\d u}{u}+\frac{\d v}{v},\]
whence we get $\d(uv)=u\,\d v+v\,\d u$ by multiplying both sides by $uv$.

In the general case $u,v\in L^\infty(\mm)$, choose a constant $C>0$ so big that
$u+C,v+C\geq 1$. By the case treated above, we know that
\begin{equation}\label{eq:calc_rules_d_aux1}\begin{split}
\d\big((u+C)(v+C)\big)&=(u+C)\,\d(v+C)+(v+C)\,\d(u+C)\\
&=(u+C)\,\d v+(v+C)\,\d u\\
&=u\,\d v+v\,\d u+C\,\d(u+v),
\end{split}\end{equation}
while a direct computation yields
\begin{equation}\label{eq:calc_rules_d_aux2}
\d\big((u+C)(v+C)\big)=\d\big(uv+C(u+v)+C^2\big)
=\d(uv)+C\,\d(u+v).
\end{equation}
By subtracting \eqref{eq:calc_rules_d_aux2} from \eqref{eq:calc_rules_d_aux1},
we finally obtain that $\d(uv)=u\,\d v+v\,\d u$, as required.
This completes the proof of the Lebniz rule iii).
\end{proof}
\bigskip

\noindent{\bf Acknowledgements.} 
\noindent This research has been supported by the MIUR SIR-grant `Nonsmooth Differential Geometry' (RBSI147UG4).
{\footnotesize
\def\cprime{$'$} \def\cprime{$'$}

}
\end{document}